\documentclass [12pt,reqno]{amsart}
\usepackage{amsmath}
\usepackage{amsthm}
\usepackage{dsfont}
\newtheorem{thm}{Theorem} 
\usepackage{amsfonts, amssymb}
\usepackage[utf8]{inputenc} %% mac coding
\usepackage[english]{babel}
\usepackage{mathtools}
\usepackage{ mathrsfs }
\usepackage{enumerate}
\usepackage{ amssymb }
\usepackage{cite}

\newtheorem{theorem}{Theorem}
\newtheorem{Proposition}[theorem]{Proposition}

\newtheorem*{theorem*}{Theorem}

\newtheorem{lemma}[thm]{Lemma}

\newtheorem{corollary}[thm]{Corollary}
\theoremstyle{definition}

\theoremstyle{definition}

\theoremstyle{definition}

\numberwithin{equation}{section}

\def\R{{\mathbb R}}
\def\Z{{\mathbb Z}}
\def\N{{\mathbb N}}

\def\wt{\widetilde}

\def\<{\langle}
\def\>{\rangle}

\def \aa {{\alpha}}
\def \ll {{\lambda}}
\def \LL {{\Lambda}}
\def \H {{\mathds{H}}}

\begin{document}

%\title{$\bold{S^1}$ has Markov type $\bold 2$ with constant $\bold 1$}
\title{Dimension of a snowflake of a finite Euclidean subspace}
%\title{Examples of spaces which have Markov type $\bold 2$ with constant $\bold 1$}

%\title{Markov type $\bold 2$ with constant $\bold 1$}

\begin{abstract}
Let $X$ be an $n$-point subset of a Euclidean space and $0 < a < 1$. The classical theorem of Schoenberg implies that the snowflake space $X^a$ can be isometrically embedded into Euclidean space. 
In the paper we show that points in the image of such an embedding always are in general position.  
As application we prove the analogue of Schoenberg's result for quotients of Euclidean spaces by finite groups. 
\end{abstract}

\author{Vladimir Zolotov}
\address[Vladimir Zolotov]{Steklov Institute of Mathematics, Russian Academy of Sciences, 27 Fontanka, 191023 St.Petersburg, Russia, University of Cologne, Albertus-Magnus-Platz, 50923 Köln, Germany and Mathematics and Mechanics Faculty, St. Petersburg State University, Universitetsky pr., 28, Stary Peterhof, 198504, Russia.}
%\email[Vladimir Zolotov]{paranuel@mail.ru}

\maketitle
% ----------------------------------------------------------------

\section{Introduction}

For $n \in \N$ we denote by $E^n$ an $n$-dimensional Euclidean space and by $\H$ a separable Hilbert space.  In the paper we discuss topics related to
the following result by Schoenberg.

\begin{Proposition}[\cite{SCH}, Corollary 1] \label{ShoThm}
Let $X$ be an $n$-point subspace of $\H$. Then for $0 < \aa < 1$ the snowflake space $X^{\aa}$ can be isometrically embedded into $E^{n-1}$. 
\end{Proposition}

We prove the following theorem which claims that points in the image of the snowflake space $X^{\aa}$ from the Proposition \ref{ShoThm} are in general position.  

\begin{theorem} \label{DimThm}
Let $X$ be an $n$-point subspace of $H$. Then for $0 < \aa < 1$ the snowflake space $X^{\aa}$  cannot be isometrically embedded into $E^{n-2}$. 
\end{theorem}
Theorem \ref{DimThm} answers the question by  H. Maehara see \cite{M}, Problem 2.8. 

In order to state an analogue of Proposition \ref{ShoThm} for quotients of Euclidean space by finite groups we have to define certain quotient metric spaces $Q(n,G)$. Which we use as target spaces for isometric embeddings of snowflakes in this setting. 
Let $G$ be a finite group and $n \in \N$. We denote the group ring of $G$ consisting of all functions $f:G \rightarrow \R$ by $\R[G]$. We consider $\R[G]$ 
as a Euclidean space with a standard Euclidean structure. % a vector space  $\R[G]^n$ by $L(n,G)$, its subspace 
We denote $\{(f_1,\dots,f_n) \in \R[G]^n \vert \sum_{i = 1}^n\sum_{g \in G}f_i(g) = 1\}$ by  $L_1(n,G)$. We consider an action of $G$ on  $L_1(n,G)$ induced by regular action on its coordinates, we denote the corresponding metric quotient $L_1(n,G)/G$  by $Q(n,G)$.  

%We obtain the following corollary of Theorem \ref{DimThm} which can be considered as an analogue of Proposition \ref{ShoThm} for quotients of Euclidean spaces by finite groups.

\begin{corollary}\label{Cor}
Let $X$ be an $n$-point metric space and let $G$ be a finite group acting on Euclidean space $E^m$ by isometries. Suppose that $X$ can be isometrically embedded into the metric quotient $E^m/G$. Then, for every $0 \le \aa \le 1$ the snowflake $X^\aa$ can be isometrically embedded into $Q(n,G)$. 
\end{corollary}

%\section{Proof of Theorem \ref{DimThm}}
\section{Inequalities of negative type}

Let $x_1,\dots,x_n$ be points in a metric space $X$, $D$ be an $n \times n$ matrix  given by $D_{ij} = d(x_i,x_j)^2$ and $\LL = (\ll_1,\dots \ll_n)$ be such that $\sum \ll_i  = 0$. The inequality of the form $\LL D \LL^T \le 0$ is called an inequality of the negative type.

The proof of the Proposition \ref{ShoThm} is based on the following characterization of embeddability into $H$ in terms of inequalities of the negative type (see \cite{SCH}, \cite{SCH2} or \cite{SEI}).

\begin{Proposition}[\cite{SCH}, Section 3, Formula 5] \label{NegTypeEmbedd}
Let $X = \{x_1, \dots, x_n\}$ be a finite metric space and $D$ be an $n \times n$ matrix $D$ given by $D_{ij} = d(x_i,x_j)^2$.  $X$ can be isometrically embedded into $H$, iff  for every $\LL = (\ll_1,\dots \ll_n)$  such that $\sum \ll_i  = 0$ we have $\LL D \LL^T \le 0$.
\end{Proposition}

The proof of Theorem \ref{DimThm} basically can be found in the proof of Proposition \ref{ShoThm} modulo the following lemma providing a 
geometric characterization of inequalities of the negative type.   

\begin{lemma}\label{GeomNeg}
Let $x_1,\dots,x_k, x_{k+1},\dots,x_l$ be a family of points in Euclidean space $E^m$, $\ll_1,\dots ,\ll_k, \ll_{k+1},\dots \ll_l \in \R$ be such that $\ll_1,\dots, \ll_k \ge 0$, $\sum_{1}^{k}\ll_i = 1$, $\ll_{k+1},\dots, \ll_l \le 0$, $\sum_{k+1}^{l}\ll_i = -1$, $\LL = (\ll_1,\dots,\ll_l)$ and $D$ be an  $n \times n$  matrix given by $D_{ij} = d^2(x_i,x_j)$. Then $\LL D \LL^T  = -2 \vert x_+ x_- \vert^2$, where $x_+ = \sum_1^k \ll_i x_i$, $x_- = \sum_{k+1}^l \ll_i x_i$,.
\end{lemma}
\begin{proof}
For the case $l = 2k$, $\ll_1 = \dots = \ll_k = -\ll_{k+1} = \dots = -\ll_{l} = \frac{1}{n}$ see \cite{DM}, Theorem 1. By taking certain points multiple times 
we deduce the case $\ll_i = \frac{a_i}{n}$, $a_i \in \Z$ from the previous one.
Finally, the general case follows by the limiting procedure. 
\end{proof}

The following corollary is a direct implication of Lemma \ref{GeomNeg}.

\begin{corollary}\label{CorGeomNeg}
Points $x_1,\dots, x_l \in E^m$ are not in general position, iff there exists  $\LL = (\ll_1,\dots,\ll_l) \ne 0$ s.t., $\sum_{i=1}^{l}\ll_i = 0$ and  $\LL D \LL^T = 0$, where $D$ is an $l \times l$ matrix given by $D_{ij} = d^2(x_i,x_j)$.
\end{corollary}

\section{Proof of Theorem \ref{DimThm}}

We need the following two lemmas which are parts of the proof of Proposition \ref{ShoThm}.

\begin{Proposition}[\cite{SCH}, Section 3, Formula 8]\label{ShoFormula}
For $0 < \aa < 2$ and $t > 0$ we have $$t^{\aa} =  c\big(\frac{\aa}{2}\big) \int_0^\infty(1 - e^{-\ll^2t^2})\ll^{-1-\aa}d\ll,$$
where $c(\aa) > 0$ is a certain function of $\aa$.  
\end{Proposition}

In particular if $x_1, x_2$ are point in a metric space $X$ then we have  $$|x_1x_2|^{2\aa} = c(\aa) \int_0^\infty(1 - e^{-\ll^2\vert x_1 x_2 \vert^2})\ll^{-1-2\aa}d\ll,$$ for every  $0 < \aa < 1$.

\begin{Proposition}[\cite{SCH}, Section 3, Formula 7]\label{ShoPos}
For $x_1,\dots,x_n \in \H$, $\LL = (\ll_1,\dots,\ll_n)$ and $S$ an $n \times n$ matrix given by $S_{ij}= e^{-\ll^2 \vert x_ix_j\vert^2}$, we have 
$\LL S \LL^T \ge 0$. 
\end{Proposition}

Propositions  \ref{ShoThm} and \ref{NegTypeEmbedd} imply that inequalities of the negative type hold in a snowflake of a Euclidean space. 
The following lemma claims that they also hold strict. 

\begin{lemma}\label{StrictNeg}
For every distinct $x_1,\dots,x_n  \in \H$ and every $\ll_1, \dots,\ll_n \in \R$ s.t. $\sum \ll_i = 0$ we have 
$\LL D^\aa \LL^T < 0$, where $0 < \aa < 1$ and $D^\aa$ is an $n \times n$ matrix given by $D_{ij} = \vert x_i x_j \vert^{2\aa}$.
\end{lemma}

\begin{proof}
Suppose that the statement of the Lemma is not true. I.e, there exists a family of distinct points $x_1, \dots, x_n \in H$ and  $\ll_1, \dots,\ll_n \in \R$ s.t. $\sum \ll_i = 0$ and $0 < \aa < 1$ satisfying $\LL D^\aa \LL^T = 0$.

Propositon \ref{ShoFormula} implies that $\LL S \LL^T = 0$, where the matrix $S = D^\aa $ is given by  ${S_{ij}= c(\aa) \int_0^\infty (1 - e^{-\ll^2 \vert x_ix_j\vert^2}) \ll^{-1-2\aa}d\ll}$. Hence we have ${\int_0^\infty \LL (-\wt S(\ll)) \LL^T \ll^{-1-2\aa}d\ll = 0},$ where $\wt S(\ll)$ is an $n \times n$ matrix given by $\wt S(\ll)_{ij} = e^{-\ll^2 \vert x_ix_j\vert^2}$.

Proposition \ref{ShoPos} provides that $\LL \wt S(\ll) \LL^T \ge 0$ for every $\ll > 0$. Which implies that $\LL \wt S(\ll) \LL^T= 0$ for every $\ll > 0$. Applying Proposition \ref{ShoFormula} with a different $\wt \aa$ we have that $\LL D^{\wt\aa} \LL^T = 0$, for every $0 < \wt \aa < 1$. Passing to the limit gives $\LL D^{0} \LL^T = 0$. Note that $D^0$ is (square) distance matrix of a standard simplex. Thus, $\LL D^{0} \LL^T = 0$ contradicts Corollary \ref{CorGeomNeg}.
\end{proof}

\begin{proof}[Proof of Theorem \ref{DimThm}]
Now Theorem \ref{DimThm} easily follows from Lemma \ref{StrictNeg} and Corollary \ref{CorGeomNeg}.
\end{proof}

\section{Proof of Corollary \ref{Cor}}
\begin{proof}
Without lost of generality we can assume that $X$ is a subset of $\R^m/G$. We also assume that points of $X$ are in general position, the result for degenerate point configurations follows by a limiting procedure. We remind that $|X| = n$. Let $\rho:\R^m \rightarrow \R^m/G$ be a quotient map.

%\textit{Step 1:} 
At first we are going to show that for every $0 \le \aa < 1$ there exists an isometric action of $G$ on $\R^{\vert G\vert n-1}$ such that $X^\aa$ embeds into the quotient metric space $Q_\aa  = \R^{|G|n-1}/G$ isometrically. 

Let $Y = \rho^{-1}(X) \subset \R^m$, since the points of $X$ are in general position, we have $|Y| = |G|n$. Group $G$ acts on $Y$ and we have $X = Y/G$ and $X^\aa = Y^\aa/G$, $0 \le \aa \le 1$. By Proposition \ref{ShoThm} metric space $ Y^\aa$ can be isometrically embedded into $E^{n-1}$. We fix this embedding and identify $Y^\aa$ with its image. Note that by Theorem \ref{DimThm} the convex hull of $Y^\aa$ has full dimension. Thus an isometric action of $G$ on $Y^\aa$ can be extended to an isometric action on $E^{n-1}$. We can take $Q_\aa$ to be a corresponding quotient space. So $X^\aa$ embeds isometrically into $Q_\aa$.

In the remaining part of the proof we are going to show that $X^\aa$ embeds isometrically into $Q(n,G)$.
Our plan is the following we show that  $Q_0$ is isometric to $Q(n,G)$. Then we give a continuity argument which implies that for every $0 \le \aa <1$ the quotient spaces $Q_\aa$ should be isometric to $Q_0 = Q(n,G)$. We remark that if one is not interested in optimizing the dimension of a target space then he can use a simpler argument to construct a target space say $\wt Q(n,G)$.

%\textit{Step 2:}
 Let $x_1,\dots,x_n$ be points of $X$, $y_1,\dots,y_{n|G|}$ be points of $Y$ numbered s.t., 
$$\{y_1,\dots,y_{|G|}\} = \rho^{-1}(x_1),$$ $$\{y_{|G| + 1},\dots,y_{2|G|}\} = \rho^{-1}(x_2),$$ $$\dots$$ $${\{y_{(n-1)|G| - 1},\dots, y_{n|G|}\} = \rho^{-1}(x_n)}.$$
Consider a map of $Y$ into $\R^{n|G|}$ given by $y_i \mapsto e_i$, where $e_1,\dots,e_{n|G|}$ is  a standard basis of $\R^{n|G|}$. Let $L := \{x \in \R^{n|G|}| \sum_{i = 1}^{n|G|}x_i = 1 \}$ denotes the hyperplane in $\R^{n|G|}$.  Note that action of $G$ on $Y$ induces an action of $G$ on $L$.
We denote the corresponding quotient set by  $L/G$.

For every $0 \le \aa < 1$ the distance structure $Y^{\aa}$ on a set $Y$ induce a scalar product on $L$. We denote the corresponding Euclidean space by $L_{\aa}$. This euclidean structures induce a continuous family of quotient metrics $L_{\aa}/G$ on the set $L/G$. Note that $L_{\aa}/G$ is isometric to $Q_\aa$. In the case $\aa = 0$ the metric $L_0$ is the standard Euclidean metric on $L$. Moreover by examining the action of $G$ on $L$ one can see that 
$L_{0}/G$ is isometric to $Q(n,G)$.

\begin{lemma}\label{finite}
Isometric classes of quotient metrics on $L/G$ are finite in number.
\end{lemma}
\begin{proof}
Every quotient metric on $L/G$ is induced by a representation $\rho:G \rightarrow O(L)$. The representation $\rho$ can be decomposed into direct sum of irreducible representations 
$\rho = \bigoplus \rho_i$.  Since there exist only finite number of inequivalent irreducible representations it suffices to show that for a pair of representations  $\rho, \wt \rho:G \rightarrow O(L)$ having the same decomposition  into direct sum of irreducible representations we have $$\wt \rho = T^{-1} \circ \rho \circ T, $$
where $T$ is an orthogonal operator.
It easy to see that a pair of equivalent irreducible representations  $\rho_i, \wt \rho_i$ there exists an orthogonal operator $T_i$, s.t.
  $$\wt \rho_i = T_i^{-1} \circ \rho_i \circ T_i. $$
  Finally we can take $T = \bigoplus T_i$.
\end{proof}

%To finish the proof we need the following statement.
%\begin{lemma}
%Suppose that $M_1$ and $M_2$ are separable metric spaces with fixed .  
%\end{lemma}

To show that for every $0 \le \aa < 1$ the quotient space $L_\aa/G$  is isometric to $Q(n,G)$ we going to use the concept of noncompact Gromov-Hausdorff Limit (see \cite{BBI}). First we have to introduce some relevant terminology. A pointed metric space is a pair $(X,p)$
of a metric space and a point $p \in X$. A pair of pointed metric spaces  $(X,p)$,  $(Y,q)$ are sad to be isometric if there exists an bijective isometry $f:X \rightarrow Y$ s.t., $f(p) = q$. A metric space is called boundedly compact if all its closed bounded subsets are compacts. 

 Now suppose by contradiction that there exists $0 < \aa < 1$ s.t. $L_\aa/G$ is not isometric to $L_0/G = Q(n,G)$.
Then there exist a pair of boundedly compact pointed metric spaces $(X, p)$, $(Y, q)$ and a set of numbers  $0 \le \wt \aa, \aa_1, \aa_2, \aa_3, \dots < 1$ s.t., 
\begin{enumerate}
\item{$(X,p)$ is not isometric to $(Y,q)$.}\label{NotIso}
\item{$(X,p)$ is isometric to $(L_{\wt\aa}/G, 0)$,}
\item{$(Y,q)$ is isometric to $(L_{\aa_i}/G, 0)$,}
\item{$\aa_i \underset{i \rightarrow \infty}{\rightarrow} \wt \aa.$}
\end{enumerate}

The continuity of the family $L_\aa/G$ implies that  $(L_{\aa_i}/G, 0)$ converge to $(L_{\wt\aa}/G, 0) \cong (X,p)$ in the Gromov-Hausdorff sense. This sequence is also  obviously converges to $(Y,q)$. By \cite{BBI}, Theorem 8.1.7 we have that $(X,p)$ is isometric to $(Y,q)$, which contradicts (\ref{NotIso}).

%Let $Q_1$ and $Q_2$ be two quotient metric spaces with underline set $L/G$.  Suppose that $Q_1$ and $Q_2$ are arbitrary close 
%in the following sense there exist a set of maps $f_n:Q_1 \rightarrow Q_2$ such that every $f_n$ maps $0$ to $0$ and $\sup_{x \ne y}{\frac{|f_n(x)f_n(y)|}{|xy|}} \underset{n\rightarrow \infty}{\rightarrow} 1$, $\inf_{x \ne y}{\frac{|f_n(x)f_n(y)|}{|xy|}}\underset{n\rightarrow \infty}{\rightarrow} 1$. 

%Since $Q_1$, $Q_2$ are separable by the standard diagonal argument (see for example \cite{BBI} proof of Theorem 7.3.30) there exist isometric maps $f:Q_1 \rightarrow Q_2$ and 
%$g:Q_2 \rightarrow Q_1$ which map $0$ into $0$. Consider a map $g \circ f:Q_1 \rightarrow Q_1$ it maps isometrically $B(0, r)$ into $B(0, r)$.
%Since $B(0, r)$ is a compact this map is onto.  Thus $f$ also maps  $B(0, r)$ onto $B(0, r)$. Hence $f$ is an isometric map onto $Q_2$.
\end{proof}

\subsection*{Acknowledgements}
I thank Sergey V. Ivanov and Alexander Lytchak for advising me during different stages of the work. 
I'm grateful to Mikhail Basok, Christian Lange and Nina Lebedeva for fruitful discussions.
 The paper is supported by the Russian Science Foundation under grant 16-11-10039.

\bibliography{circle}
\bibliographystyle{plain}

\end{document}